\documentclass[a4paper,11pt,reqno]{amsart}
\usepackage{amsmath,amssymb,amsthm,amsfonts,mathrsfs,amsopn} 
\usepackage{enumitem,dsfont}
\usepackage{hyperref}
\usepackage{mathtools}
\mathtoolsset{showonlyrefs}  
\usepackage{xcolor}
\usepackage{tikz-cd}
\usepackage{amssymb}

\usepackage{xcolor}

\usepackage{geometry}
\newtheorem{thm}{Theorem}
\newtheorem{lem}[thm]{Lemma}
\newtheorem{prop}[thm]{Proposition}

\theoremstyle{definition}
\newtheorem{defin}[thm]{Definition}

\newtheorem{rem}[thm]{Remark}

\newcommand{\C}{\mathbb{C}}
\newcommand{\R}{\mathbb{R}}

\newcommand{\D}{\mathbb{D}}
\newcommand{\N}{\mathbb{N}}
\newcommand{\Z}{\mathbb{Z}}

\newcommand{\cL}{\mathcal{L}}
\newcommand{\cM}{\mathcal{M}}

\newcommand{\cD}{\mathcal{D}}

\newcommand{\eps}{\varepsilon}

\newcommand{\sign}{\text{sign}}

\newcommand{\be}{\begin{equation}}
\newcommand{\ee}{\end{equation}}

\usepackage[normalem]{ulem}
\definecolor{DarkGreen}{rgb}{0,0.5,0.1} 

\newcommand\soutD{\bgroup\markoverwith
{\textcolor{DarkGreen}{\rule[.5ex]{2pt}{1pt}}}\ULon}
\newcommand{\Hm}[1]{\leavevmode{\marginpar{\tiny%
$\hbox to 0mm{\hspace*{-0.5mm}$\leftarrow$\hss}%
\vcenter{\vrule depth 0.1mm height 0.1mm width \the\marginparwidth}%
\hbox to
0mm{\hss$\rightarrow$\hspace*{-0.5mm}}$\\\relax\raggedright #1}}}

\title[On the continuum limit for a model of binary waveguides arrays]{On the continuum limit for a model of binary waveguide arrays}

\author{William Borrelli}
\address[W. Borrelli]{Dipartimento di Matematica e Fisica, Universit\`a Cattolica del Sacro Cuore, Via dei Musei 41, I-25121, Brescia, Italy.}
\email{william.borrelli@unicatt.it}
\urladdr{}

\begin{document}
\maketitle


\begin{abstract}
In this paper we prove the convergence of solutions to discrete models for binary waveguide arrays toward those of their formal continuum limit, for which we also show the existence of localized standing waves. This work rigorously justifies formal arguments and numerical simulations present in the Physics literature.
\end{abstract}
\medskip

{\footnotesize

\medskip

}


\section{Introduction and main results}
Waveguide arrays (WAs) have fundamental photonic properties that allow to describe various classical optical phenomena, such as diffractive properties or the existence of solitons. Such systems have also been successfully used in both theoretical investigations and experimental verifications of optical phenomena predicted by non-relativistic quantum mechanics, see \cite{Tran-solitons} and references therein.

Recently, \emph{binary} waveguide arrays (BWAs) have been proposed to simulate quantum relativistic effects. Those arrays consist of two different alternating types of waveguides in which light propagation can be described, in suitable regimes, by a discrete model corresponding to  discrete linear/nonlinear Dirac equations. Formal arguments in the Physics literature show that the continuum limit of such models is given by  Dirac equations, strongly suggesting the possibility to use BWAs to simulate various phenomena, including \emph{Zitterbewegung} \cite{longhi} (the trembling motion of a free electron), the Klein paradox and the existence of Dirac solitons \cite{Tran-solitons}. In recent works, the optical analogue of special states, known in quantum field theory as \emph{Jackiw-Rebbi states} (J-R states) \cite{Tran-JR}, are shown to exist at the interface of two BWAs with opposite masses. Pairs of vertically displaced BWAs have been proposed to simulate neutrino oscillations \cite{marini-neutrino}, corresponding to two coupled cubic Dirac equations in the continuum limit. The systems described above are essentially two dimensional, and the transverse spatial variable formally plays the role of time in the model. Then various dynamical features of the evolution can be `read' from the properties of the system along the transverse direction. This strongly indicates that BWAs can be used to efficiently simulate relativistic phenomena in optical systems.
\medskip

The aim of the present article is to rigorously prove the convergence of solutions to discrete models of BWAs describing Dirac solitons and J-R states toward those of their continuum limit, for which we also prove the existence of localized standing waves, numerically investigated in the literature.
\medskip

Let us describe the models under consideration.

In the monochromatic regime, light propagation in the BWAs is described \cite{Tran-solitons} by the following discrete equation
\be\label{eq:discreteeq2}
\imath \frac{d a_n(z)}{dz}=-k[a_{n+1}(z)-a_{n-1}(z)]+(-1)^n\beta a_n(z)-\gamma\vert a_n(z)\vert^2a_n(z)\,,
\ee
where $a_n(z): \R\to \C$ is the $n$-th waveguide electric field amplitude, $z\geq 0$ is the longitudinal spatial coordinate, and $2\beta$ and $k$ are the propagation mismatch and the coupling coefficient of adjacent waveguides, respectively, while $\gamma\in \R$ is the nonlinear coefficients of waveguides. In what follows we assume $\gamma=1$, corresponding to a self-focusing medium. Here $\beta,k$ are suitable functions. Notice that in \eqref{eq:discreteeq2} $z$ formally plays the role of time.
\smallskip

Equation \eqref{eq:discreteeq2} can be recast into a discrete nonlinear Dirac equation, as follows.
Set 
\be\label{eq:spinor}
\left\{\begin{array}{rcl}
	\psi^1_h(x_n,\cdot) &=& (-1)^na_{2n}(\cdot)\,,\\
	\psi^2_h(x_n,\cdot)&= & \imath (-1)^na_{2n-1}(\cdot)\,,
\end{array}\right. \qquad n\in\N\,.
\ee
Consider the position of the WAs on the lattice $h\Z$, with spacing $h>0$, the points being denoted by $x_n$, where $x_n=nh$, $n\in\Z$.

Using \eqref{eq:discreteeq2} we thus get, for $n\in\N$, $z\geq0$,
\be\label{eq:discretenld}
\displaystyle
\left\{\begin{array}{rcl}
	\imath\frac{ d\psi^1_h}{dz}(x_n,z) &=& -\imath k[\psi^2_h(x_{n+1},z)-\psi^2_h(x_n,z)]+\beta\psi^1_h(x_n,z)-\vert\psi^1_h(x_n,z)\vert^2\psi^1_h(x_n,z)\,,\\
	\imath\frac{ d\psi^2_h}{dz}(x_n,z) &= & -\imath k[\psi^1_h(x_n,z)-\psi^1_h(x_{n-1},z)]-\beta\psi^2_h(x_n,z)-\vert\psi^2_h(x_n,z)\vert^2\psi^2_h(x_n,z)\,,
\end{array}\right.\,.
\ee
 so that, taking 
 \[
 k(h)=\frac{1}{h}
 \] we find 
\be\label{eq:discretenld2}
\imath \frac{d\psi_h}{dz}(x_n,z)=\D_h\psi_h(x_n,z)+\beta\sigma_3\psi_h(x_n,z)-G(\psi_h)\psi_h(x_n,z)\,,\qquad n\in\N\,,
\ee
where 
\be\label{eq:discretedirac}
\D_h=-\imath\begin{pmatrix}0 & \partial_h \\ \partial^*_h & 0  \end{pmatrix}\,,\qquad \sigma_3=\begin{pmatrix}1 & 0 \\ 0 & -1 \end{pmatrix}\,,\qquad G(\psi)=\begin{pmatrix} \vert \psi_1(x_n,z)\vert^2 & 0 \\ 0 & \vert \psi_2(x_n,z)\vert^2 \end{pmatrix}\,.
\ee
Here 
\be\label{eq:discretederivative}
\partial_h\varphi:=h^{-1}[\varphi(x_{n+1})-\varphi(x_n)] \,,\qquad \partial^*_h\varphi:=h^{-1}[\varphi(x_n)-\varphi(x_{n-1})]
\ee
are the discrete right-hand derivative operator and its adjoint on $\ell^2(h\Z,\C^2)$, respectively.

Then, letting $h\to0^+$ we formally get the following cubic Dirac equation on the real line 
\be\label{eq:cubicdirac}
\imath \partial_z\Psi(x,z)=\cD\Psi(x,z)+\beta\sigma_3\Psi(x,z)- G(\Psi)\Psi(x,z)\,,
\ee
for $(x,z)\in\R\times(0,\infty)$, where $\sigma_3=\begin{pmatrix}1 &0 \\ 0 & -1 \end{pmatrix}$ is the third Pauli matrix and 
\be\label{eq:dirac}
\cD=-\imath\sigma_1\partial_x=-\imath\begin{pmatrix}0 & \partial_x \\  \partial_x & 0 \end{pmatrix}
\ee
 is the one-dimensional Dirac operator.

We consider both the case of a positive constant mass $\beta=const.$ and of a \emph{domain wall} mass interpolating between two different asymptotic values.
\begin{defin}\label{def:domainwall}
An increasing $C^1$ function $\beta:\R\to\R$ with bounded derivative is called a \emph{domain wall} if it is odd, i.e. $\beta(x)=-\beta(x)$, and there holds
\be\label{eq:betalimits}
\lim_{x\to\infty}\beta(x)=\beta(\infty)\in(0,+\infty)\,,\qquad \lim_{x\to-\infty}\beta(x)=-\beta(\infty)<0 \,.
\ee
Moreover, the asymptotic limits are approached sufficiently rapidly
\[
\int^\infty_0\vert \beta(y)-\beta(\infty)\vert\,dy<\infty\,,\qquad \int^0_{-\infty}\vert \beta(y)+\beta(\infty)\vert\,dy<\infty
\]
\end{defin} 
An example of such functions is given by $\beta(x)=\tanh(x)$, $x\in\R$. Notice that our definition is slightly more restrictive than the one given in \cite{LWW}, for technical reasons, but it covers physically interesting cases.
\smallskip

When $\beta$ is constant or a domain wall, the Dirac operator $\cD+\beta\sigma_3$ is self-adjoint on $L^2(\R,\C^2)$, with domain $H^1(\R,\C^2)$. When $\beta$ is constant the spectrum is purely continuous and it is given by 
\[
\operatorname{Spec}(\cD+\beta\sigma_3)=(-\infty,-\beta]\cup[\beta,+\infty)\,,
\]
see \cite{diracthaller}.
 The case of a domain wall mass has been treated in \cite{LWW}, where the authors shows that, in addition to the continuous spectrum, finitely many simple eigenvalues can appear in the mass gap 
 \[
 \operatorname{Spec}(\cD+\beta\sigma_3)=(-\infty, \beta(-\infty)]\cup\{\lambda_i\}^N_{i=1}\cup[\beta(\infty),+\infty)\,.
 \]
 In particular, when $\beta(\cdot)=\tanh(\cdot)$, there is only one eigenvalue $\lambda=0$.
 
 We remark that the results in \cite{LWW} also apply to non smooth functions like the signum function
 \be\label{eq:signum}
 \sign (x)=\left\{\begin{array}{rcl}
	1 \,,\qquad x>0\\
	-1 \,,\qquad x<0
\end{array}\right. \,,
 \ee
 which is the one actually considered in \cite{Tran-JR} for the investigation of photonic J-R states. Here we allow for general domain wall functions, as in Def. \ref{def:domainwall}.
 
 More details on the Dirac operators involved and on the functional setting can be found in Section \ref{sec:prel}.
 
 \medskip
 Before stating our main results, some definitions are needed.

 Given a function $f\in L^1_{loc}(\R)$, its discretization $f_h:h\Z\to\C$ is defined as
\be\label{eq:discretization}
f_h(x_n):=\frac{1}{h}\int^{x_{n+1}}_{x_n}f(x)\,dx\,.
\ee
The piecewise constant interpolation $q_hf_h$ is given by
\be\label{eq:piecewiseconstant}
q_hf_h(x):=f_h(x_n)\,,\qquad x\in[x_n,x_{n+1})\,.
\ee
We define the piecewise linear interpolation $p_hu_h:\R\to\C$ of a lattice function $u_h:h\Z\to\C$ setting
\be\label{eq:piecewiselinear}
(p_hu_h)(x):=u_h(x_n)+(\partial_hu_h)(x_n)(x-x_n)\,,\quad x\in[x_n,x_{n+1})\,,
\ee
where $\partial_h v(x):=h^{-1}[v(x+h)-v(x)]$ is the discrete right-hand side derivative operator.
\smallskip

We first consider the case of a constant mass $\beta$.

\begin{thm}\label{thm:main}
Take $\chi\in H^1(\R,\C^2)$, and consider its discretization $\chi_h:h\Z\to\C^2$ defined as in \eqref{eq:discretization}. Let $\psi_h\in C^1_z([0,\infty), L^2_h(\R^2,\C^2)))$ be the unique global solution to \eqref{eq:discretenld} with initial datum $\chi_h$.
Then there exist a constant $C>0$, independent of $\chi_h$ and $h$, such that for every $0<T<(2C\Vert\varphi_h\Vert^2_{H^1_h})^{-1}$, there holds
\[
p_h\psi_h\rightharpoonup\Psi\,,\qquad\mbox{weakly-$\ast$ in} \,\,L^\infty([0,T],H^1(\R,\C^2)) \qquad \mbox{as $h\to0^+$\,,} 
\]
where $\Psi\in C^0([0,\infty), H^1(\R,\C^2))$ is the unique global solution to \eqref{eq:cubicdirac} with initial datum $\chi$.
\end{thm}
\begin{rem}
Observe that in the above Theorem we can only consider timescales depending on the size of the initial datum, in contrast to \cite{KLS} where the result holds for all fixed $T>0$. This difference is due to the sign indefiniteness of the energy associated with the Dirac equation which can not be used to obtain uniform estimates for solutions, contrary to what can be done for Schr\"odinger type equations. The same observation, of course, applies to Theorem \ref{thm:main3}.
\end{rem}

The second main result concerns the existence of standing waves for \eqref{eq:cubicdirac} for frequencies in the mass gap. Such solutions are of the form
\be\label{eq:standingwaves}
\Psi(z,x)=e^{-\imath \omega z}\Phi(x)\,,\qquad \omega\in\R\,,
\ee
that is, they are `time' modulation of a fixed spatial profile $\Phi:\R\to\C^2$. As already remarked, in the model considered the oscillating factor corresponds to spatial periodicity in the transverse direction. The function $\Phi(\cdot)$ thus solves the following stationary equation
\be\label{eq:stationary}
(\cD+\beta\sigma_3-\omega)\Phi(x)=G(\Phi(x))\Phi(x)\,,\qquad x\in\R.
\ee
\begin{thm}\label{thm:main2}
For every $\omega\in(0,\beta)$, equation \eqref{eq:stationary} admits a smooth solution of the form $\Phi=(u,\imath v)$, with $u,v$ real-valued, which decays exponentially at infinity
\be\label{eq:decay}
\vert \Phi(x)\vert=\sqrt{u^2(x)+v^2(x)}\leq C e^{\sqrt{\beta^2-\omega^2}\vert x\vert}\,,\qquad \forall x\in\R\,,
\ee
for some constant $C=C(\beta,\omega)>0$.
\end{thm}
The proof of the above theorem relies on a phase-plane analysis, which allows to get a rather precise description of the shape of the solution.
\smallskip

The next results are the analogue of Theorem \ref{thm:main} and Theorem \ref{thm:main2}, when $\beta$ is a domain wall function.

\begin{thm}\label{thm:main3}
Take an initial datum $\chi\in H^1(\R,\C^2)$ and a domain wall function $\beta\in C^1(\R)$. Consider their discretization $\chi_h:h\Z\to\C^2$, $\beta_h:h\Z\to\R$ defined as in \eqref{eq:discretization}. Let $\psi_h\in C^1_z([0,\infty), L^2_h(\R^2,\C^2)))$ be the unique global solution to \eqref{eq:discretenld} with initial datum $\chi_h$ and mass $\beta_h$.
Then, there exist a constant $C>0$, independent of $\chi_h$ and $h$, such that for every $0<T<(2C\Vert\varphi_h\Vert^2_{H^1_h})^{-1}$for every $0<T<+\infty$, there holds
\[
p_h\psi_h\rightharpoonup\Psi\,,\qquad\mbox{weakly-$\ast$ in} \,\,L^\infty([0,T],H^1(\R,\C^2)) \qquad \mbox{as $h\to0^+$\,,} 
\]
where $\Psi\in C^0([0,\infty), H^1(\R,\C^2))$ is the unique global solution to \eqref{eq:cubicdirac} with initial datum $\chi$.
\end{thm}

We mention that in the literature both dynamical systems and variational methods have been successfully used to prove the existence of standing waves for nonlinear Dirac equations, see e.g. \cite{shooting,CCM,es} and references therein. In particular, a variational approach allows to prove an analogous result for equation \eqref{eq:cubicdirac} when $\beta$ is a domain wall function, as in that case the phase plane analysis used for Theorem \ref{thm:main2} does not seem to be applicable.

\begin{thm}\label{thm:main4}
For every $\omega\in(0,\beta(\infty))$, $\omega\notin \operatorname{Spec}(\cD+\beta\sigma_3)$, equation \eqref{eq:stationary} admits a smooth solution exponentially decaying at infinity
\be\label{eq:decay2}
\vert \Phi(x)\vert\leq C e^{-\sqrt{\beta^2(\infty)-\omega^2}\vert x\vert}\,,\qquad \forall x\in\R\,,
\ee
for some constant $C=C(\beta(\infty),\omega)>0$.
\end{thm}
The proof of Theorem \ref{thm:main} and Theorem \ref{thm:main3} follows \cite{KLS}, where the authors prove a similar result for discrete cubic nonlinear Schr\"odinger equations with long-range lattice interactions, whose solutions converge to (possibly fractional) Schr\"odinger equations on the real line. 

We mention that, generally speaking, linear and nonlinear Dirac equations in low dimensions recently attracted a considerable attention in the mathematical literature, starting from the works \cite{FWhoneycomb,wavedirac,FWwaves} on two dimensional honeycomb structures. Natural nonlinear models, as in the present paper, are given by cubic Dirac equations, which are Sobolev critical. Consistency of the effective model and the existence of stationary solutions have been investigated in \cite{arbunichsparber}, \cite{CCM,shooting,borrellifrank}. We mention that, to the best of our knowledge, global well-posedness of the Cauchy problem is not known in that case for Kerr-type nonlinearities, as available results rely on the null-structure of the nonlinearity \cite{bh,bc}, which is absent for Kerr-type nonlinear terms.


\section{Preliminaries}\label{sec:prel}
We collect here notions and basic definitions useful in the sequel.
\subsection{Functional spaces}
Consider the lattice $h\Z$, with spacing $h>0$. We denote 
by $L^2_h$ the Hilbert space of sequences $u_h:h\Z\to\C$ for which the norm $\Vert u_h\Vert_{L^2_h}$ is finite, endowed with the inner product
\[
\langle u_h,v_h\rangle_{L^2_h}:=h\sum_{n\in\Z}\overline{u_h(x_n)}v_h(x_n),.
\]
For $u_h\in L^2_h$, its Fourier transform $\widehat{u_h}:[-\pi,\pi]\to\C$ is defined as
\[
\widehat{u_h}(\xi):=\frac{1}{\sqrt{2\pi}}\sum_{n\in\Z}u_h(x_n)e^{-\imath n\xi}\,,
\]
and there holds $\widehat{u_h}\in L^2([-\pi,\pi])$. The Fourier inversion formula holds
\be\label{eq:inversion}
u_h(x_n)=\frac{1}{\sqrt{2\pi}}\int^\pi_{-\pi}\widehat{u_h}(\xi)e^{\imath n\xi}\,d\xi\,.
\ee
The Sobolev norm is given by
\[
\Vert u_h\Vert^2_{H^1_h}:=h\int^\pi_{-\pi}(1+h^{-2}\vert \xi\vert^2)\vert\widehat{u_h}(\xi)\vert^2\,d\xi,,
\]
and the Sobolev space $H^1_h$ is defined accordingly. The following discrete uniform Sobolev inequality holds.
\begin{lem}
There holds
\[
\Vert u_h\Vert_{L^\infty_h}\leq \frac{\sqrt 2}{2}\Vert u_h\Vert_{H^1_h}\,,
\]
for every $u_h\in H^1_h$, with $\Vert u_h\Vert_{L^\infty_h}:=\sup_{n\in\Z}\vert u_h(x_n)\vert$.
\end{lem}
\begin{proof}
By \eqref{eq:inversion} and using the Cauchy-Schwarz inequality we find
\[
\begin{split}
\Vert u_h\Vert_{L^\infty_h}&\leq \frac{1}{\sqrt{2\pi}}\int^\pi_{-\pi}\vert\widehat{u_h}(\xi)\vert\,d\xi\leq \frac{1}{\sqrt{2\pi}}\left(h^{-1}\int^\pi_{-\pi}\frac{d\xi}{1+h^{-2}\vert\xi\vert^2} \right)^{1/2}\Vert u_h\Vert_{H^1_h} \\
&=\frac{1}{\sqrt{2\pi}} \left(\int^\infty_{-\infty}\frac{d\xi}{1+\vert\xi\vert^2} \right)^{1/2}\Vert u_h\Vert_{H^1_h}\leq \frac{\sqrt 2}{2} \Vert u_h\Vert_{H^1_h}\,.
\end{split}
\]
\end{proof}
\subsection{Discretization and interpolation}
Recall the definitions \eqref{eq:discretization} , \eqref{eq:piecewiseconstant} and \eqref{eq:piecewiselinear}.
The following estimates can be found in \cite[Sec. 3.3]{KLS}
\begin{lem}\label{lem:estimates}
For any $f\in H^1(\R)$ there holds
\[
\Vert f_h\Vert_{H^1}\leq C\Vert f\Vert_{H^1}\,,\quad \Vert q_h f_h\Vert_{L^2}\leq\Vert f_h\Vert_{L^2_h}\,,\quad\Vert p_hf_h\Vert_{H^1}\leq C\Vert f_h\Vert_{H^1_h}\,,
\]
for some constant $C>0$.
\end{lem}
The above definitions and results immediately extend to vector-valued functions.
\subsection{Dirac operators.}
The discrete Dirac operator $\D_h$ in \eqref{eq:discretedirac} is easily seen to be self-adjoint operator on $L^2_h$, and as such it generates a strongly continuous unitary one-parameter group $(e^{-\imath z\D_h})_{z\in\R}$.

Similarly, the operator $\cD$ in \eqref{eq:dirac} is self-adjoint on $L^2(\R,\C^2)$, with domain and form-domain given by the Sobolev spaces $H^1(\R,\C^2)$ and $H^{1/2}(\R,\C^2)$, respectively. Such properties follows from the fact that the matrix-valued symbol of the operator is
\[
\widehat{\cD}(\xi)=\begin{pmatrix}0 & \xi \\ \xi &0  \end{pmatrix}\,,
\]
see \cite{diracthaller}. As anticipated in the Introduction, when $\beta$ is a positive constant, the spectrum 
\[
\operatorname{\cD+\beta\sigma_3}=(-\infty,-\beta]\cup[\beta,+\infty)\,,
\]
is purely absolutely continuous, with a spectral gap due to the mass term.

If $\beta$ is a domain wall function (see Def. \ref{def:domainwall}), additional finitely many simple eigenvalue can appear in the gap
\[
\operatorname{Spec}(\cD+\beta\sigma_3)=(-\infty,-\beta]\cup[\beta,+\infty)\,,
\]
besides the purely continuous spectrum. We remark that in \cite{LWW} the authors prove that when $\beta(\cdot)=\tanh(\cdot)$, $\lambda=0$ is the unique eigenvalue in the gap.

\section{Estimates for the discrete equation}
The following proofs apply to both the case of a constant mass and of a domain wall. In the former $\beta>0$ is constant, while in the latter $\beta=\beta_h$ is the discretization of a domain wall function, see \eqref{eq:discretization} and Def.\ref{def:domainwall}.
\subsection{Well-posedness}
 \begin{prop}\label{prop:discretegwp}
For every initial datum $\varphi_h\in L^2_h$ there exists a unique global classical solution $\psi_h\in C^1([0,\infty), L^2_h)$ to \eqref{eq:discretenld}, with $\psi_h(0,\cdot)=\varphi_h(\cdot)$. Moreover, the norm 
\[
N(\psi_h):=\Vert \psi\Vert_{L^2_h}
\]
is conserved.
\end{prop}
\begin{proof}
Rewrite \eqref{eq:discretenld} in Duhamel form 
\be\label{eq:discreteDuhamel}
\psi_h(z)=e^{-\imath z(\D_h+\beta\sigma_3)}\varphi_h-\imath\int^z_0 e^{-\imath(z-s)(\D_h+\beta\sigma_3)}G(\psi_h(x))\psi_h(s)\,ds\,.
\ee
The map $\psi_h\mapsto G(\psi_h)\psi_h$ is locally Lipschitz in $L^2_h$, by the embedding $L^2_h\subseteq L^\infty_h$. A standard fixed-point argument gives local well-posedness in $L^2_h$, that is, there exists $0<T\leq+\infty$ and a solution $\psi_h\in C^1([0,T), L^2_h)$, with $\psi_h(0)=\varphi_h$. Global extendability of solutions follows from the conservation of the norm, arguing as follows. Assume that $T<+\infty.$ It suffices to show that the $L^2_h$-norm of the solution remains bounded as $z\to T^-$. Suppose, for the moment, that
\be\label{eq:ell^2bound}
\sup_{t\in[0,T)}\Vert \psi_h \Vert_{L^2_h}\leq M<+\infty\,,
\ee
and let $\tau=\tau(M)$ be the lifespan provided by the local well-posedness proof. Choose $0<s\leq T$ such that $T-s<\tau$. Then the Cauchy problem \eqref{eq:cubicdirac} with initial datum $\psi_h(s)$ at $t=s$ has a local solution up to $z=s+\tau$. By uniqueness we then conclude that $\psi_h(\cdot)$ can be extended up to $z=s+\tau>T$, contradicting the maximality of $T$.

The conservation of the $L^2_h$-norm follows by a direct computation using \eqref{eq:discretenld2}. Indeed
\[
\begin{split}
\frac{d}{dz}\Vert\psi_h\Vert^2_{L^2_h}&=\Re(\partial_z\psi_h,\psi_h)_{L^2_h} \\
&=\Re(\imath(\D_h+\beta\sigma_3)\psi_h,\psi_h)_{L^2_h}-\Re\gamma(\imath G(\psi_h)\psi_h,\psi_h)_{L^2_h}\\
&=0
\end{split}
\]
where the first term on the r.h.s. vanishes by the self-adjointness of $\D_h$.
\end{proof}
\subsection{Auxiliary results}
In this section we collect results used in the proof of the main Theorem. We preliminarily state the following nonlinear Gronwall-type result \cite{bihari}, that we use to obtain a priori estimates for solutions to \eqref{eq:discretenld}.
\begin{lem}[Bihari inequality]\label{lem:bihari}
Let $u,f:[0,\infty)\to[0,\infty)$ be continuous functions and $w:[0,\infty)\to[0,\infty)$ continuous non-decreasing such that $w(u)>0$ on $(0,\infty)$. If there holds
\be\label{eq:ineq}
u(t)\leq M+\int^t_0f(s)w(u(s))\,ds\,, \qquad t\geq0
\ee
where $M\geq0$ is a constant, then
\be\label{eq:biharibound}
u(t)\leq G^{-1}\left(G(M)+\int^t_0 f(s)\,ds \right)\,,\qquad t\in[0,T].
\ee
The function $G$ is defined by 
\[
G(x)=\int^x_{x_0}\frac{dy}{w(y)}\,,\qquad x\geq0,x>0\,,
\]
and $G^{-1}$ is its inverse, while $T>0$ is such that
\be\label{eq:inversedomain}
G(M)+\int^t_0 f(s)\,ds \in\operatorname{Dom}(G^{-1})\,,\qquad \forall t\in[0,T]\,.
\ee
\end{lem}
\begin{lem}\label{lem:apriori}
Let $\psi\in C^1([0,\infty),L^2_h)$ be a solution to \eqref{eq:discretenld}, with $\psi_h(0,\cdot)=\varphi_h(\cdot)$. Then there exists $C>0$ such that for any $0<T<(2C\Vert \varphi_h\Vert^2_{H^1_h})^{-1}$ there holds
\be\label{eq:aprioriH^1}
\sup_{z\in[0,T]}\Vert\psi_h(z) \Vert_{H^1_h}\leq A(T,\Vert\varphi_h\Vert_{H^1_h})\,,
\ee
where $A(T,\Vert \varphi\Vert_{H^1_h})=\frac{1}{(\Vert\varphi\Vert^{-2}_{H^1_h}-2CT)^{1/2}}$
\end{lem}
\begin{proof}
Fix $T>0$, to be chosen later, and consider the Duhamel formula \eqref{eq:discreteDuhamel}. Notice that the following Parseval identity holds
\[
\Vert \psi_h\Vert^2_{L^2}+\Vert\partial_h\psi_h\Vert^2_2=h\int^\pi_{-\pi}(1+4\sin^2\xi)\vert\widehat{\psi_h}(\xi)\,\vert^2\,d\xi
\]
as a direct computation using the definition of the Fourier transform shows that
\[
\widehat{\partial_h\psi_h}(\xi)=2\imath(\sin\xi)\widehat{\psi_h}(\xi)\,.
\]

Then, being a Fourier multiplier, the propagator $(e^{-\imath z\D_h+\beta\sigma_3})_{z\in\R}$ preserves the $H^1_h$ norm. Then we get
\[
\Vert \psi_h(z)\Vert_{H^1_h}\leq\Vert\varphi_h\Vert_{H^1_h}+\int^z_0\Vert  G(\psi_h)\psi_h\Vert_{H^1_h}\,ds\,.
\]
Recall that $G(\psi)\sim\vert\psi\vert^2$. Moreover, the following Leibniz rule holds
\[
\partial_h(fg)(x_n)=(\partial_h f)(x_n)g(x_{n+1})+f(x_n)\partial_hg(x_n)\,,
\]
so that 
\[
\partial_h(G(\psi_h)\psi_h)(x_n)\sim(\partial_h \vert\psi_h\vert^2)(x_n)\psi_h(x_{n+1})+\vert\psi_h\vert^2(x_n)\partial_h\psi_h(x_n)\,.
\]
The embedding $L^\infty_h\hookrightarrow H^1_h$  thus implies that $H^1_h$ is an algebra and 
\[
\Vert G(\psi_h)\psi_h\Vert_{H^1_h}\leq C\Vert\psi_h\Vert^3_{H^1_h}\,,
\]
for some constant $C>0$. Then we get
\be\label{eq:applybihari}
\Vert \psi_h(z)\Vert_{H^1_h}\leq \Vert\varphi_h\Vert_{H^1_h}+C\int^z_0\Vert \psi_h(s)\Vert^3_{H^1_h}\,ds\,.
\ee
We can now apply Lemma \ref{lem:bihari} with $M:=\Vert\varphi_h\Vert_{H^1_h}$, $f:=C$ and $w(u)=u^3$. In this case
\[
G(x)=\frac{1}{2}\left(\frac{1}{x^2_0}-\frac{1}{x^2} \right)\,,\qquad G^{-1}(y)=\sqrt{\frac{1}{\frac{1}{x^2_0}-2y}}
\]
and we choose $x_0=\Vert\varphi_h\Vert_{H^1_h}$, so by \eqref{eq:inversedomain} we get 
\[
T<(2C\Vert \varphi_h\Vert^2_{H^1_h})^{-1}
\]
and the a priori bound \eqref{eq:aprioriH^1}.
\end{proof}

The next result relates the discrete Dirac operator $\D_h$ in \eqref{eq:discretedirac} to the continuous operator $\cD$ in \eqref{eq:dirac}, in the limit as $h\to 0^+$.
\begin{lem}\label{lem:convergenceoperator}
For every $\varphi\in C^\infty_c(\R,\C^2)$ there holds
\[
\D_h\varphi\to\cD\varphi\,,
\]
in $L^2(\R)$, as $h\to0^+$.
\end{lem}
\begin{proof}
Recall the definition of $\D_h$ in \eqref{eq:discretedirac}. Then a simple computation using the Fourier transform gives
\[
\mathcal F(\partial_h\varphi)(\xi)=\left[\frac{e^{\imath h \xi}-1}{h\xi}\right]\xi\widehat{\varphi}(\xi)\,,\quad \mathcal F(\partial^*_h\varphi)(\xi)=\left[\frac{1-e^{-\imath h \xi}}{h\xi}\right]\xi\widehat{\varphi}(\xi)\,.
\]
Observe that 
\[
\frac{e^{\imath z}-1}{z}=\frac{z}{2}+o(z)+\imath \sin z\,,\quad\frac{1-e^{-\imath z}}{z}=-\frac{z}{2}+o(z)+\imath \sin z\,,\quad\mbox{as $z\to0$}\,,
\]
and then 
\[
\mathcal F(\D_h\varphi)(\xi)\to\mathcal F(\cD\varphi)(\xi)\,, 
\]
pointwise for all $\xi\in\R$, as $h\to0^+$. Since $\xi\widehat{\varphi}(\xi)\in L^2(\R,\C^2)$ as $\varphi$ is smooth and compactly supported, the desired result
\[
\lim_{h\to0^+}\int_\R \vert\D_h\varphi - \cD\varphi\vert^2\,dx=0
\]
follows by dominated convergence.
\end{proof}
We conclude the section with another technical result \cite[Chapter VI, Lemma 4.1]{Lady}.
\begin{lem}\label{lem:interpolationconvergence}
For any $f\in L^2(\R,\C^2)$ and $g\in H^1(\R,\C^2)$, there holds
\[
\Vert p_hf_h-f\Vert_{L^2}\to0\quad\mbox{and}\quad \Vert p_hg_h-g\Vert_{H^1}\qquad \mbox{as $h\to0^+$.}
\]
\end{lem}
\section{Well-posedness of the limit model} 
In this section we prove global well-posedness for the Cauchy problem for \eqref{eq:cubicdirac}. To our knowledge, the basic idea of the proof can be traced back to \cite{goodman}, see also \cite{pely}.
\begin{prop}\label{eq:gwp}
Let $\chi\in H^1(\R,\C^2)$. Then \eqref{eq:cubicdirac} admits a unique global solution $\Psi\in C^0([0,\infty),H^1(\R,\C^2))\cap C^1([0,\infty),L^2(\R,\C^2))$.
\end{prop}
\begin{proof}
Consider the Duhamel integral formula
\be\label{eq:continuousDuhamel}
\Psi(z)=e^{-\imath z\cD}\chi-\imath\int^z_0e^{-\imath (z-s)\cD}[-\beta\sigma_3\Psi(s)+G(\Psi(s))]\Psi(s)\,ds\,.
\ee
The embedding $H^1(\R,\C^2)\hookrightarrow L^\infty(\R,\C^2)$ ensures that the map $\Psi\mapsto G(\Psi)\Psi$ is locally Lipschitz on $H^1(\R,\C^2)$, so that local well-posedness follows by standard contraction mapping arguments (see the proof of Prop. \eqref{prop:discretegwp}). Then there exists a maximal lifespan $0<T\leq+\infty$, $T=T(\Vert \chi\Vert_{H^1})$, and a unique solution $\Psi\in L^\infty([0,T],H^1(\R,\C^2))\cap W^{1,\infty}([0,T],L^2(\R,\C^2))$. By standard arguments (see, for instance, \cite[Lemma 2.2]{JDEgrafi}) regularity can be improved, so that $\Psi\in C^0([0,T],H^1(\R,\C^2))\cap C^{1}([0,T],L^2(\R,\C^2))$, and 
\be\label{eq:initialdatum}
\Psi(0)=\chi\in H^1(\R,\C^2)\,.
\ee
 
\smallskip

As for the proof of Proposition \ref{prop:discretegwp}, it suffices to show that the $H^1$-norm of the solution remains bounded as $t\to T^-$, that is 
\be\label{eq:H^1bound}
\sup_{t\in[0,T)}\Vert \Psi(t) \Vert_{H^1(\R,\C^2)}\leq M<+\infty\,.
\ee

Define
\[
\varphi:=\Psi^1+\Psi^2\,,\qquad \eta:=\Psi^1-\Psi^2\,,
\]
where $\Psi:=(\Psi^1,\Psi^2)^T:\R\times[0,T)\to\C^2$, so that substituting into \eqref{eq:cubicdirac} we find
\[
\left\{\begin{array}{rcl}
-\imath \partial_z\varphi&=&-\imath\partial_x\varphi+\beta\eta-\gamma[\frac{1}{4}(\vert \varphi\vert^2)+\vert\eta\vert^2)\varphi+\frac{1}{2}\Re(\varphi\overline{\eta})\eta] \\
-\imath \partial_z\eta&=&\imath\partial_x\eta+\beta\varphi-\gamma[\frac{1}{4}(\vert \varphi\vert^2)+\vert\eta\vert^2)\eta+\frac{1}{2}\Re(\varphi\overline{\eta})\varphi]
\end{array}\right. \,.
\]
Here $\Re(\cdot)$ denotes the real part of a complex number. Multiply the first line by $\overline{\varphi}(\vert\varphi\vert^2+\vert\eta\vert^2)^p$ and the second line by $\overline{\eta}(\vert\varphi\vert^2+\vert\eta\vert^2)^p$, $p\geq1$, adding the two equations and taking the imaginary part gives 
\[
\frac{1}{p+1}\partial_z(\vert\varphi\vert^2+\vert\eta\vert^2)^{p+1}\leq \frac{1}{p+1}\partial_x(\vert\varphi\vert^2-\vert\eta\vert^2)^{p+1}+\imath\beta(\eta\overline{\varphi}-\varphi\overline{\eta})(\vert\varphi\vert^2-\vert\eta\vert^2)^{p+1}\,.
\]
Then integrating in the $x$-variable and using that $2\vert \varphi\vert\vert\eta\vert\leq\vert\varphi\vert^2+\vert\eta\vert^2$ we find
\[
\partial_z \int_\R (\vert\varphi\vert^2+\vert\eta\vert^2)^{p+1}\,dx\leq 4(p+1)\beta\int_\R (\vert\varphi\vert^2+\vert\eta\vert^2)^{p+1}\,dx\,.
\]
Observing that by the following equivalence $\int_\R (\vert\varphi\vert^2+\vert\eta\vert^2)^{p+1}\,dx\simeq \Vert (\varphi,\eta)\Vert^{2p+2}_{L^{2p+2}_x}$, the Gronwall's lemma implies that 
\be\label{eq:lpbound}
\Vert (\varphi,\eta)(z)\Vert_{L^{2p+2}_x}\leq e^{Cz}\Vert (\varphi,\eta)(0)\Vert_{L^{2p+2}_x}\,,\qquad \forall z\in[0,T)\,,
\ee
for some constant $C>0$. Since \eqref{eq:lpbound} holds for all $p\geq1$, we also get the $L^\infty$-bound
\be\label{eq:uniformbound}
\Vert (\varphi,\eta)(z)\Vert_{L^{\infty}_x}\leq e^{Cz}\Vert (\varphi,\eta)(0)\Vert_{L^{\infty}_x}\,,\qquad \forall z\in[0,T)\,.
\ee
Analogous estimates thus hold for $\Psi$.
\smallskip

A simple computation shows that, given $\xi:\R\to\C^2$, there holds $\Vert \cD\xi\Vert^2_{L^2}=\Vert \partial_x\xi\Vert^2_{L^2}$. Recall that $e^{\imath z\cD}$ is an $L^2$-isometry and that it commutes with $\cD$, so that applying $\cD$ to both sides of \eqref{eq:continuousDuhamel} we get
\[
\Vert \partial_x\Psi (z)\Vert_{L^2_x}\leq \Vert \partial_x\chi\Vert_{L^2_x}+\int^z_0 \Vert \partial_x [G(\Psi)\Psi](s)\Vert_{L^2_x}+\Vert \beta'\Psi(s)\Vert_{L^2_x}+\Vert\beta\partial_x\Psi(s)\Vert_{L^2_x}\,ds\,,\quad \forall z\in[0,T)\,.
\]
Since $\beta\in C^1$ with bounded derivative, we get
\[
\Vert \partial_x\Psi (z)\Vert_{L^2_x}\leq \Vert \partial_x\chi\Vert_{L^2_x}+CT+\int^z_0 \Vert \partial_x [G(\Psi)\Psi](s)\Vert_{L^2_x}+C\Vert\partial_x\Psi(s)\Vert_{L^2_x}\,ds\,,\quad \forall z\in[0,T)\,,
\]
where we have used the conservation of the $L^2$-norm. 
Observe that the integral involving $G$ contains terms of the schematic form $\vert\Psi\vert^2\partial_x\Psi$, and then by \eqref{eq:uniformbound} there exists a constant $B(T)>0$ such that
\[
\Vert \partial_x\Psi (z)\Vert_{L^2_x}\leq \Vert \partial_x\chi\Vert_{L^2_x}+CT+ B(T)\int^z_0 \Vert \partial_x \Psi(s)\Vert_{L^2_x}\,ds\,,\quad \forall z\in[0,T)\,.
\]
Another application of the Gronwall's lemma, combined with \eqref{eq:lpbound}, finally gives 
\[
\Vert \Psi(z)\Vert_{H^1_x}\leq Ae^{Cz}\Vert \chi\Vert_{H^1_x}\,,\qquad \forall z\in[0,T)\,,
\]
for some finite constants $A,C>0$ depending on $T$. Then \eqref{eq:H^1bound} follows, as we assumed $0<T<\infty$.
\end{proof}

\section{Proof of the main results for constant mass}
\subsection{The discrete-to-continuum limit}
Recall the assumptions of Theorem \eqref{thm:main}. Let $\chi\in H^1(\R,\C^2)$ and consider its discretization $\chi_h:h\Z\to\C^2$, defined as in \eqref{eq:discretization}. By Lemma \ref{lem:interpolationconvergence} $p_h\chi_h\to \chi$ in $H^1$ as $h\to0^+$. Now denote by $\psi_h\in C^1([0,\infty),L^2_h)$ the global solution to \eqref{eq:discretenld} with initial datum $\chi_h$, provided by Proposition \ref{prop:discretegwp}.
\smallskip

Fix $0<T< (2C\Vert \chi_h\Vert^2_{H^1_h})^{-1}$ as in Theorem \ref{thm:main}. We first prove some bounds for $p_h\psi_h$ and $\partial_zp_h\psi_h$, uniform for $z\in[0,T]$. Notice that, since we are interested in the limit $h\to0^+$ we will assume $0<h<h_0$, for a fixed $h_0>0$.
\begin{lem}\label{lem:uniformbounds}
There holds
\be\label{eq:uniformbounds}
\sup_{z\in[0,T]}\Vert p_h\psi_h(z)\Vert_{H^1}\leq C\,,\qquad \sup_{z\in[0,T]}\Vert \partial_zp_h\psi_h(z)\Vert_{L^2}\leq C\,,
\ee
where the constant $C>0$ depends only on $h_0,T$ and 
\[
M:=\sup_{0< h\leq h_0}\Vert p_h\chi_h\Vert_{H^1}\leq C\,.
\]
\end{lem}
\begin{proof}
Consider $M>0$ as above, and observe that $M<+\infty$ as $p_h\chi_h\to \chi$ in $H^1$ as $h\to0^+$. Recall that by Lemma \ref{lem:estimates} we have
\[
\Vert p_h\psi_h(z)\Vert_{H^1}\leq C\Vert \psi_h(z)\Vert_{H^1_h}\,,
\]
and the a-priori estimate in Lemma \ref{lem:apriori} gives
\[
\sup_{z\in[0,T]}\Vert\psi_h(z)\Vert_{H^1_h}\leq C(T,\Vert \chi_h\Vert_{H^1_h})\,.
\]
Combining those facts with $\Vert\chi_h\Vert_{H^1_h}\leq C\Vert \chi\Vert_{H^1}$ from Lemma \ref{lem:estimates} we get the first item in \eqref{eq:uniformbounds}.

By \eqref{eq:discretenld} we find
\[
\Vert\partial_z\psi_h(z)\Vert_{L^2_h}\leq \Vert \D_h\psi_h(z)\Vert_{L^2_h}+\Vert G(\psi_h(z))\psi_h(z)\Vert_{L^2_h}\,.
\]
There holds, using the definition of $\D_h$
\[
\Vert \D_h\psi_h(z)\Vert_{L^2_h}\leq C\sup_{z\in[0,T]}\Vert \psi_h(z)\Vert_{H^1_h}\leq C(T,\Vert\chi_h\Vert_{H^1_h})
\]
by the previous bounds. Moreover, we have
\[
\Vert G(\psi_h(z))\psi_h(z)\Vert_{L^2_h}\leq \Vert G(\psi_h(z))\psi_h(z)\Vert_{H^1_h}\,,
\]
and by the Leibniz rule in $H^1_h$ and the embedding $H^1_h\hookrightarrow L^\infty_h$ we then get
\[
\Vert G(\psi_h(z))\psi_h(z)\Vert_{L^2_h}\leq C(T,\Vert\chi_h\Vert_{H^1_h})\,,
\]
as $G(\varphi)\varphi\sim \vert\varphi\vert^2\varphi$ is a cubic nonlinearity. We have thus proved that 
\[
\sup_{z\in[0,T]}\Vert \partial_z\psi_h(z)\Vert_{L^2_h}\leq C\,,
\]
for some constant $C>0$, depending on $h_0,T, M$. The second inequality in \eqref{eq:uniformbounds} follows observing that $\partial_z$ and $p_h$ commute.
\end{proof}
We are now in a position to prove the first main result of the paper.
\begin{proof}[Proof of Theorem \ref{thm:main}]
By Lemma \ref{lem:uniformbounds} and the Banach-Alaoglu theorem there exists a sequence $h_n\to0^+$, as $n\to\infty$, such that
\be\label{eq:psiconvergence}
p_{h_n}\psi_{h_n}\rightharpoonup\Psi \quad \mbox{weakly-$*$ in $L^\infty([0,T],H^1(\R,\C^2))$}
\ee
and
\be\label{eq:derivativepsiconvergence}
\partial_zp_{h_n}\psi_{h_n}\rightharpoonup\Psi \quad \mbox{weakly-$*$ in $L^\infty([0,T],L^2(\R,\C^2))$}\,,
\ee
as $n\to\infty$. We now prove that $\Psi$ solves \eqref{eq:cubicdirac}. To this aim, it suffices to prove that the equation is verified for a.e. $z\in[0,T]$. By \eqref{eq:derivativepsiconvergence} we have
\be\label{eq:derivativetest}
\lim_{n\to\infty}\int^T_0\langle\Phi,\imath\partial_zp_{h_n}\psi_{h_n}(z)\rangle_{L^2}\,dz=\int^T_0\langle\Phi,\imath\partial_z\Psi(z)\rangle_{L^2}\,dz\,,
\ee
for every $\Phi\in L^1([0,T], H^1(\R,\C^2))$. 

We now need to prove that
\be\label{eq:test}
\lim_{n\to\infty}\int^T_0\langle\Phi,p_{h_n}\D_{h_n}\psi_{h_n}(z)\rangle_{L^2}\,dz=\int^T_0\langle\Phi,\cD\Psi(z)\rangle_{L^2}\,dz\,,
\ee
for every $\Phi\in L^1([0,T], H^1(\R,\C^2))$. We prove that claim, by density, for functions $\Phi(t,x)=f(t)u(x)$ with $f\in C^\infty_c([0,T])$ and $u\in C^\infty_c(\R,\C^2)$. Let $u_h$ be the discretization of $u$, as in \eqref{eq:discretization}. By Lemma \ref{lem:interpolationconvergence} $\Vert p_hu_h-u\Vert_{H^1}\to0$ as $h\to0^+$, and there holds $\Vert p_h\D_h\psi_h(z)\Vert_{L^2}\leq C$, so that 
\[
\langle p_{h_n}u_{h_n}-u,p_{h_n}\D_{h_n}\psi_{h_n}(z)\rangle_{L^2}\leq C\Vert p_{h_n}u_{h_n}-u\Vert_{H^1}\to0\,,
\]
as $n\to+\infty$, for every $t\in[0,T]$. Then it suffices to prove that
\be\label{eq:interpolatedtest}
\lim_{n\to\infty}\int^T_0\langle fp_{h_n}u_{h_n},p_{h_n}\D_{h_n}\psi_{h_n}(z)\rangle_{L^2}\,dz=
\int^T_0\langle fu,\cD\Psi(z)\rangle_{L^2}\,dz\,.
\ee
A straightforward computation using the definition of $\D_h$ in \eqref{eq:discretedirac} gives
\[
\langle fp_{h_n}u_{h_n},p_{h_n}\D_{h_n}\psi_{h_n}(z)\rangle_{L^2}=\langle fp_{h_n}\D_{h_n}u_{h_n},p_{h_n}\psi_{h_n}(z)\rangle_{L^2}
\]
for every $z\in[0,T]$. We claim that 
\be\label{eq:interpolatedconv}
p_{h_n}\D_{h_n}u_{h_n}\to\cD u\,,
\ee
in $L^2(\R,\C^2)$ as $n\to\infty$.

A direct computation shows that 
\[
(\D_{h_n}u_{h_n})(x_n)=(\D_{h_n}u)_{h_n}(x_n)
\]
 that is, the discretization \eqref{eq:discretization} and the action of $\D_h$ commute. Moreover, by Lemma \ref{lem:convergenceoperator} there holds
 \be\label{eq:Dconv}
 \Vert \D_{h_n}u-\cD u\Vert_{L^2}\to0\,,\qquad\mbox{as $h\to0^+$}
 \ee
and then we find
\[
\begin{split}
\Vert p_{h_n}\D_{h_n}u_{h_n}-\cD u\Vert_{L^2}&=\Vert p_{h_n}(\D_{h_n}u)_{h_n}-\cD u\Vert_{L^2} \\
&\leq \Vert p_{h_n}(\D_{h_n}u-\cD u)_{h_n}\Vert_{L^2} +\Vert p_{h_n}(\cD u)_{h_n}-\cD u\Vert_{L^2}\,.
\end{split}
\]
Since $\Vert p_{h_n}(\D_{h_n}u-\cD u)_{h_n}\Vert_{L^2}\leq\Vert \D_{h_n}u-\cD u\Vert_{L^2}$, \eqref{eq:interpolatedconv} follows combining \eqref{eq:Dconv} and Lemma \ref{lem:interpolationconvergence}. By \eqref{eq:interpolatedconv}, the dominated convergence theorem gives \eqref{eq:interpolatedtest}.

We now deal with the nonlinear term of the equation. We need to prove that
\be\label{eq:nonlinearconv}
\lim_{n\to\infty}\int^T_0\langle\Phi,p_{h_n}(G(\psi_{h_n}(z))\psi_{h_n})\rangle_{L^2}\,dz=\int^T_0\langle\Phi, G(\Psi(z))\Psi(z)\rangle_{L^2} \,dz
\ee
for every $\Phi\in L^1([0,T],H^1(\R,\C^2))$. We consider again functions $\Phi(t,x)=f(t)u(x)$ with $f\in C^\infty_c([0,T])$ and $u\in C^\infty_c(\R,\C^2)$.

Observe that
\[
\Vert p_{h_n}(G(\psi_{h_n}(z))\psi_{h_n}) \Vert_{L^2}\leq C\Vert\vert\psi_{h_n}(z)\vert^2\psi_{h_n}(z) \Vert_{L^2_h}\leq C\Vert\psi_{h_n}(z)\Vert^2_{L^\infty_{h_n}}\Vert\psi_{h_n}(z)\Vert_{L^2_{h_n} }\leq C\,,
\]
where we have used Lemma \ref{lem:apriori}. Then we can assume that $p_{h_n}(G(\psi_{h_n}(z))\psi_{h_n})$ weakly converges in $L^2(\R,\C^2)$, for a.e. $z\in[0,T]$. By \cite{Lady} this is equivalent to the weak convergence of the piecewise constant interpolation $q_{h_n}(G(\psi_{h_n}(z))\psi_{h_n})$, defined in \eqref{eq:piecewiseconstant}.

Then we are lead to show that
\be\label{eq:constantconv}
\lim_{n\to\infty}\langle u,q_{h_n}(G(\psi_{h_n}(z))\psi_{h_n})\rangle_{L^2}=\langle u, G(\Psi(z))\Psi(z)\rangle_{L^2}\,,
\ee
for a.e. $z\in[0,T]$. By \eqref{eq:psiconvergence}, $p_{h_n}\psi_{h_n}(z)\to\Psi(z)$ strongly in $L^2_{loc}(\R,\C^2)$ for a.e. $z\in[0,T]$. As a general fact, from \cite{Lady} we know that this is equivalent to the same result for the piecewise constant interpolation, namely, $q_{h_n}\psi_{h_n}(z)\to\Psi(z)$ strongly in $L^2_{loc}(\R,\C^2)$ for a.e. $z\in[0,T]$. Then notice that 
\[
q_{h_n}(G(\psi_{h_n})\psi_{h_n})=G(q_{h_n}\psi_{h_n})q_{h_n}\psi_{h_n}\,,
\]
and then since $\Vert\psi_{h_n}(z)\Vert_{L^\infty_h}\leq C$, the claim \eqref{eq:constantconv} follows by the dominated convergence theorem. This result, in turn, implies \eqref{eq:nonlinearconv}. A similar argument allows to deal with the mass term $\beta\sigma_3\Psi$ in \eqref{eq:cubicdirac}.
\smallskip

The above discussion shows that $\Psi\in L^\infty([0,T],H^1(\R,\C^2))\cap W^{1,\infty}([0,T],L^2(\R,\C^2))$ in \eqref{eq:psiconvergence},\eqref{eq:derivativepsiconvergence} satisfies
\[
\int^T_0\langle\Phi,\imath\partial_z\Psi\rangle_{L^2}\,dz=\int^T_0\langle\Psi,(\cD+\beta\sigma_3)\Psi\rangle_{L^2}\,dz-\gamma\int^T_0\langle\Phi,G(\Psi)\Psi\rangle_{L^2}\,dz\,,
\]
for every $\Phi\in L^1([0,T],H^1(\R,\C^2))$. Moreover, equation \eqref{eq:cubicdirac} is solved for a.e. $z\in[0,T]$ and by uniqueness $\Psi$ coincides with the solution found in Proposition \ref{eq:gwp}, so that the limit in \eqref{eq:psiconvergence},\eqref{eq:derivativepsiconvergence} does not depend on the sequence $h_n\to0$.
 \end{proof}
 \subsection{Existence of standing waves}
 \begin{proof}[Proof of Theorem \ref{thm:main2}]
We look for a solution to \eqref{eq:stationary} of the form 
\[
\Xi(x)=(u(x),\imath v(x))\,,\qquad x\in \R\,,
\]
with $u,v$ real-valued. Then we get the following Hamiltonian system
\be\label{eq:hamiltoniansystem}
\left\{\begin{array}{rcl}
 u'&=&-\omega v-\beta v- v^3=-\partial_v H \\
 v'&=&\omega u-\beta u +u^3=\partial_u H
\end{array}\right.
\ee
with Hamiltonian function
\be\label{eq:hamiltonian}
H(u,v)=\frac{u^4+v^4}{4}+\frac{\beta}{2}(v^2-u^2)+\frac{\omega}{2}(u^2+v^2)\,.
\ee
Here $f'=\partial_x f$.
 \begin{figure}[h]
\centering
{\includegraphics[width=.5\columnwidth]{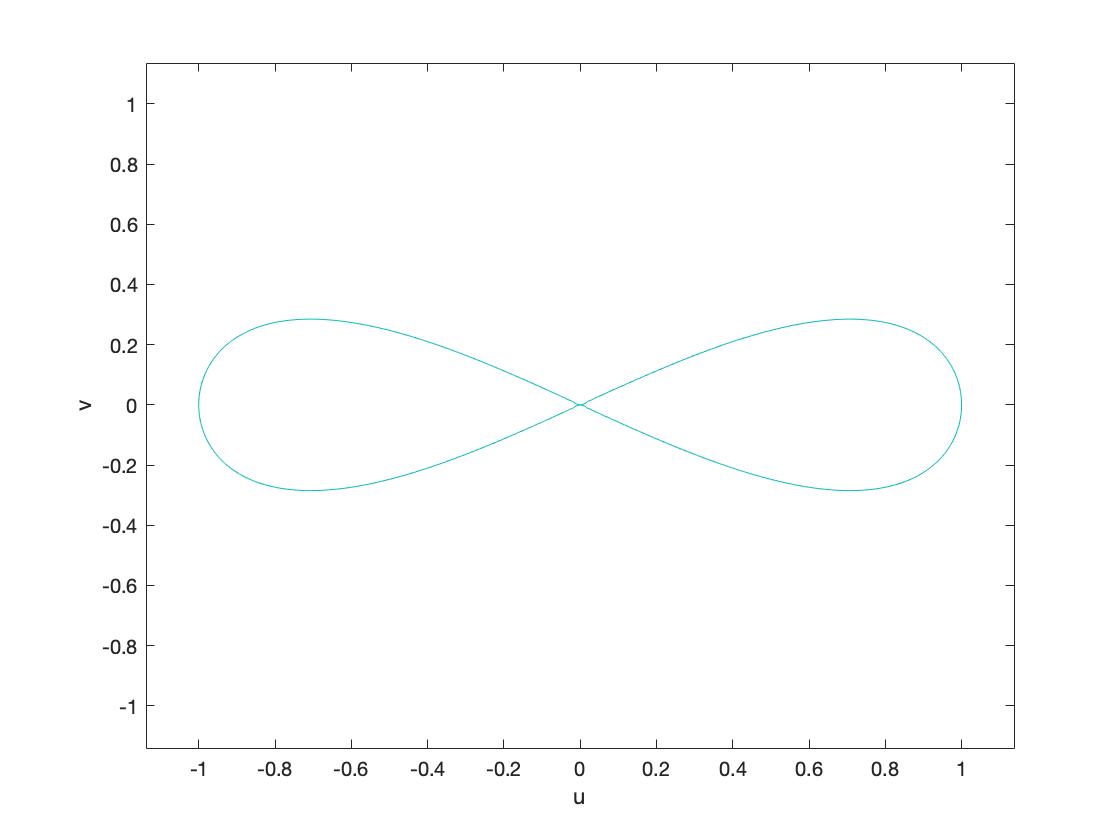}}
\caption{The level curve $\{H=0\}$ for $\beta=1,\omega=0.5$.}
\label{fig:levelset}
\end{figure}
Notice that here the variable $x$ plays the role of time. Notice that local existence for \eqref{eq:hamiltoniansystem} holds, as the system is smooth. Since we consider solutions that tend to zero as $\vert x\vert\to\infty$, we search for orbits of \eqref{eq:hamiltoniansystem} on the zero-energy curve 
\[
\{(u,v)\in\R^2\,:\,H(v,u)=0\}\,,
 \]
see Figure \ref{fig:levelset}. Such set is clearly compact, and then orbits on it exist for all times.
 
 A direct computation imposing $\partial_u H =\partial_v H=0$ give that the origin $(0,0)$ is the only equilibrium of \eqref{eq:hamiltoniansystem}. We are thus looking for a homoclinic orbit $(u,v)$
 \be\label{eq:homoclinic}
 \lim_{\vert x\vert\to\infty}(u(x),v(x))=(0,0)\,.
 \ee
 To this aim, fix as initial point on $\{H=0\}$ 
 \be\label{eq:initialdatum}
 (u(0),v(0))=(\sqrt{2(\beta-\omega)},0)\,.
 \ee The corresponding solution to \eqref{eq:hamiltoniansystem} is global, and we prove \eqref{eq:homoclinic} using the Poincar\`e-Bendixson theorem \cite[Theorem 7.16]{Teschl}.
 
 Observe that the solution with \eqref{eq:initialdatum} is confined in the region $\{u^2>v^2,u\geq 0 \}$ for all times. Indeed, it cannot cross the set $\{u^2=v^2\}$, as if $u^2=v^2$ the condition $H(u,v)=0$ gives
 \[
 u^4+2\omega u^2=0 \implies u=v=0\,,
 \]
 by \eqref{eq:hamiltonian}, and this is absurd as the origin is an equilibrium of \eqref{eq:hamiltoniansystem}. Thus
 \be\label{eq:inequality}
 u^2(x)>v^2(x)\,,\qquad\forall x\in\R\,.
 \ee
 
 Let
$$
\Omega_\pm = \{ (x,y)\in\R^2:\ \text{for some}\ x_n\to\pm\infty \,, (u(x_n),v(x_n))\to (x,y)\} \,.
$$
Since the set $\{ H=0\}$ is compact, it is not hard to see that the limit sets $\Omega_\pm$ are non-empty, compact and connected \cite[Lemma 6.6]{Teschl}. By the Poincar\'e--Bendixson theorem, one of the following three alternatives holds: (a) $\Omega_\pm$ is an equilibrium, (b) $\Omega_\pm$ is a regular periodic orbit (limit cycle), (c) $\Omega_\pm$ consists of fixed points and non-closed orbits connecting these fixed points. So we need to rule out alternative (b) and (c). Observe that in both cases the limit sets must be included in $\{H=0\}$.

Let us define the angle
\[
\theta(x):=\arctan{\frac{v(x)}{u(x)}}\,,\qquad x\in\R\,.
\]
A simple computation gives
\be\label{eq:monotone}
\theta'(x)=\frac{1}{u^2+v^2}(v'u-u'v)=\frac{u^4+v^4}{2(u^2+v^2)}\,,
\ee
using the condition $H=0$. Then $\theta'(x)>0$ if $(u,v)\neq(0,0)$, and this rules out alternative (b), as 
\[
\theta^-:=\lim_{x\to-\infty}\theta(x)<\lim_{x\to+\infty}\theta(x)=:\theta^+\,,
\]
while in presence of a limit cycle such limits do not exist. 

Suppose (c) holds. Then 
\[
\Omega_\pm\subset \{(r\cos\theta^\pm,r \sin\theta^\pm)\in\R^2\,,r \geq 0\}\,,\qquad   r^2=u^2+v^2\,,
 \]
but this is again excluded by \eqref{eq:monotone}, and \eqref{eq:homoclinic} is proved.
\smallskip
 
 We now need to prove the exponential decay estimate \eqref{eq:decay}. 

 Differentiating the first equation in \eqref{eq:hamiltoniansystem} we find
 \[
 u''=(\beta^2-\omega^2)u+F(u,v)u\,,
 \]
 with $\lim_{x\to\infty}F(u(x),v(x))=0$. Then, fixed $0<\eps<\beta^2-\omega^2$, there exists $M_\eps>0$ such that
 \[
 u''\geq (\beta^2-\omega^2-\eps)u\,.
 \]
 The function 
 \[
 U(x):=u(x)-u(M_\eps)e^{-\sqrt{\beta^2-\omega^2-\eps}(x-M_\eps)}\,,\quad x\geq M_\eps
 \]
 verifies $U(M_\eps)=0$, $\lim_{x\to+\infty}U(x)=0$ and 
 \[
 U''\geq (\beta^2-\omega^2-\eps) U\,,
 \]
 so that by the comparison principles 
 \[
 U(x)\leq 0\,,\qquad \forall x\geq M_\eps\,.
 \]
 Then 
 \[
 u(x)\leq u(M_\eps) e^{-\sqrt{\beta^2-\omega^2-\eps}(x-M_\eps)}\,,\qquad \forall x\geq M_\eps\,.
 \]
 A similar argument holds for $x\leq- M_\eps$, and then by continuity there exists a constant $C_\eps>0$ such that
 \[
 u(x)\leq C_\eps e^{-\sqrt{\beta^2-\omega^2-\eps}\vert x\vert}\,,\qquad\forall x\in\R\,.
 \]
 We can rewrite the equation for $u$ as
 \[
 -u''+(\beta^2-\omega^2)u=G\,,
 \]
 with $G(x):=-F(u(x),v(x))u(x)$, and applying the Green's function 
 \[
 u(x)=\frac{1}{2\sqrt{\beta^2-\omega^2}}\int_\R G(y)e^{-\vert x- y\vert\sqrt{\beta^2-\omega^2}}\,dy
 \]
 we get the decay estimate \eqref{eq:decay} for $u$. The same holds for $v$ by \eqref{eq:inequality} and the proof is concluded.
 \end{proof}
  \section{Proof of the main results for a domain-wall mass}
  \subsection{The discrete-to-continuum limit}
In this case the mass $\beta$ is no longer a constant, but it is assumed to be a domain-wall function (see Def. \ref{def:domainwall}). The proof is the same as for Theorem \ref{thm:main}.

\begin{proof}[Proof of Theorem \ref{thm:main3}]
Let $\beta_h$ be the discretization of the mass function $\beta$.
We only need to prove the convergence of the mass term, namely
\be\label{eq:massconv}
\lim_{n\to+\infty}\int^T_0\langle\Phi,p_{h_n}\beta_{h_n}\sigma_3\psi_{h_n}\rangle_{L^2}\,dz=\int^T_0\langle\Phi,\beta\sigma_3\Psi\rangle_{L^2}\,dz\,,
\ee
for every $\Phi\in L^1([0,T], H^1(\R,\C^2))$. As before, we prove that claim for functions $\Phi(t,x)=f(t)u(x)$ with $f\in C^\infty_c([0,T])$ and $u\in C^\infty_c(\R,\C^2)$.
\smallskip

There holds $\lim_{n\to+\infty}\Vert p_{h_n}\beta_{h_n}-\beta\Vert_{L^2}=0$. Note that $\beta$ is \emph{not} in $L^2(\R)$ as it is asymptotic to non-zero constants to $\pm\infty$. However, the difference $p_{h_n}\beta_{h_n}-\beta$ does belong to $L^2(\R)$ and the result applies. Furthermore, by \eqref{eq:psiconvergence} $p_{h_n}\psi_{h_n}\to\Psi$ in $L^2_{loc}$ as $n\to\infty$.

Then we have
\[
p_{h_n}\beta_{h_n}\sigma_3\psi_{h_n}-\beta\sigma_3\Psi=(p_{h_n}\beta_{h_n}-\beta)p_{h_n}\sigma_3\psi_{h_n}+\beta(p_{h_n}\sigma_3\psi_{h_n}-\sigma_3\Psi)
\]
so that
\[
\lim_{n\to\infty}\langle u, p_{h_n}\beta_{h_n}\sigma_3\psi_{h_n}\rangle_{L^2}=\langle u, \beta\sigma_3\Psi\rangle_{L^2}\,,
\]
 and \eqref{eq:massconv} follows.
\end{proof}
  \subsection{Existence of standing waves}
  We now turn to the proof of Theorem \ref{thm:main4}. Notice that in this case the mass $\beta$ is no longer a constant but it is a domain-wall function, as in Definition \ref{def:domainwall}. For this reason we look for solutions to \eqref{eq:stationary} using variational methods. To this aim, some preliminaries are in order. We only sketch part of the arguments as they rely on well-established techniques, referring to some references for more details.
  \smallskip
  
  Weak solutions to \eqref{eq:stationary} correspond to critical points of the $C^2$ functional
  \be\label{eq:action}
  \cL(\Phi)=\frac{1}{2}\int_\R \langle(\cD+\beta\sigma_3-\omega)\Phi,\Phi\rangle\,dx-\frac{1}{4}\int_\R\vert\Phi^1 \vert^4+\vert\Phi^2\vert^4 \,dx\,,
  \ee
  defined for $\Phi\in H^{1/2}(\R,\C^2)$, where $\Phi=(\Phi^1,\Phi^2)^T$. In order to exploit the geometry of the functional, it is more convenient to rewrite it as follows.
  
  Since $\omega\notin\operatorname{Spec}(\cD+\beta\sigma_3)$, we can decompose the Hilbert space $X=H^{1/2}(\R,\C^2)$ as 
  \be\label{eq:splitX}
  X=X^+\oplus X^-\,,
  \ee
  where $X^\pm=\operatorname{P^\pm}X$, and $\operatorname{P^\pm}$ is the spectral projector onto the positive/negative subspace of $(\cD+\beta\sigma_3-\omega)$. Accordingly, we get 
  \be\label{eq:action2}
   \cL(\Phi)=\frac{1}{2}\left(\Vert\Phi^+\Vert^2_\omega-\Vert\Phi^-\Vert^2_\omega\right)-\frac{1}{4}\int_\R\vert\Phi^1 \vert^4+\vert\Phi^2\vert^4 \,dx\,.
  \ee
  Here $\Phi^\pm:=\operatorname{P^\pm}\Phi$ and the norm $\Vert\cdot\Vert_\omega$ is defined by
  \[
  \Vert \psi\Vert_\omega:=\Vert \sqrt{\vert \cD+\beta\sigma_3-\omega\vert}\psi\Vert_{L^2}\,.
  \]
  As it can be readily seen from \eqref{eq:action2}, the functional is strongly indefinite so that in order to prove the existence of a critical point we need to show it has a \emph{linking geometry}, see e.g. \cite{BCT-SIMA,es} or \cite[Sec. II.8]{struwevariational}. 
  \begin{lem}\label{lem:negative}
  There exists $R>0$ such that if $\Phi^+\in X^+$ and $\Vert\Phi^+\Vert_\omega=1$
  \be\label{eq:negative}
  \cL(\Phi)\leq0\,,\qquad\forall\Phi\in\partial\cM_R\,,
  \ee
  where
  \[
  \partial\cM_R=\{\Phi\in\cM_R\,:\,\Vert\Psi\Vert_\omega=R\quad \mbox{or}\quad s=R\}
  \]
  and
  \be\label{eq:M}
  \cM_R:=\{\Phi=\Psi+s\Phi^+\in X^-\oplus\R_+\Phi^+\,: \Vert\Psi\Vert_\omega\leq R\,, 0\leq s\leq R \}
  \ee
  \end{lem}
  \begin{proof}
  Take $\Phi\in X^-\oplus\R_+\Phi^+\in\partial\cM_R$ as in the assumptions. Then if $\Vert\Psi\Vert_\omega\geq s$, it is immediate that $\cL(\Phi)\leq 0$. Otherwise, observe that the last integral in \eqref{eq:action2} is equivalent to $\Vert \Phi\Vert^4_{L^4}$, so we have
  \[
  \cL(\Phi)\leq \frac{1}{2}s^2-\frac{C}{4}\Vert\Phi\Vert^4_{L^4}\,,
  \]
  for some constant $C>0$. Moreover, note that \eqref{eq:splitX} induces an analogous decomposition on $L^4(\R,\C^2)$, so that 
  \[
  \Vert \psi^\pm\Vert_{L^4}\leq\Vert \psi\Vert_{L^4}\,,\qquad\forall \psi\in X\,.
  \]
  Then
  \[
  \cL(\Phi)\leq \frac{1}{2}s^2-\frac{1}{4}\Vert s\Phi^+\Vert^4_{L^4}\leq \frac{1}{2}s^2-\frac{C}{4}s^4\,,
  \]
as all norms are equivalent on the one-dimensional space $\R\Phi^+$. Since we assume $s\geq \Vert\Psi\Vert_\omega$, we must have $s=R$ and thus
\[
 \cL(\Phi)\leq \frac{R^2}{2}-C\frac{R^4}{4}<0\,,
\]
for $R>0$ large enough.
  \end{proof}
  The following lemma can be easily proved, exploiting the structure of the functional \eqref{eq:action2}.
  \begin{lem}\label{lem:positivesphere}
  There exist $r,\rho>0$ such that
  \[
  \inf_{\Phi\in S^+_r}\cL(\Phi)\geq\rho>0\,,
  \]
  where
  \be\label{eq:positivesphere}
  S^+_r:=\{\Phi\in X^+\,:\, \Vert\Phi\Vert_\omega=r\}
  \ee
  \end{lem}
  Consider the negative gradient flow $\eta_t$ of the functional $\cL$:
  \[
 \begin{cases}
\displaystyle \partial_t\eta_t=-\nabla\cL\circ\eta_t  \\[.2cm]
\displaystyle \eta_0=\operatorname{Id_X}\,.
\end{cases}.
\]
Arguing as in \cite[Lemma 2.5, Cor. 2.6]{es}, one sees that for any $t\geq0$ there holds
\[
\eta_t(\cM_R)\cap S^+_r\neq\emptyset\,.
\]
Then, combining Lemma \ref{lem:negative} and Lemma \ref{lem:positivesphere}, a deformation argument allows to prove the following
  \begin{prop}\label{prop:PS}
  Let the assumptions of Lemma \ref{lem:negative} be satisfied and define
\be\label{minmax}
c(\omega):=\inf_{t\geq0}\sup\cL(\eta_t(\cM_R)).
\ee
Then there exists a Palais-Smale sequence $\left( \Phi_{n}\right)\subset X$ at level $c(\omega)$, i.e. 
\be\label{eq:PS}
 \begin{cases}
\displaystyle \cL(\Phi_{n})\longrightarrow c(\omega)  \\[.2cm]
\displaystyle d\cL(\Phi_{n})\xrightarrow{X^{*}} 0,
\end{cases}
\qquad\mbox{as}\quad n\longrightarrow\infty \,,
\ee
and there results
\be
\label{eq:levelbound}
0<\rho\leq c(\omega)\leq\sup_{\cM_{R}}\cL<+\infty,\qquad\forall N\in\mathbb{N}.\nonumber
\ee
  \end{prop}
Let $(\Phi_n)\subset X$ be a Palais-Smale sequence for $\cL$. Standard arguments show that such sequence is bounded, see e.g. \cite[Lemma 4.1]{CCM}. Then, up to subsequences, 
  \[
   \begin{cases}
\displaystyle \Phi_n\to\Phi\,,\qquad\mbox{weakly in $X$}  \\[.2cm]
\displaystyle \Phi_n\to\Phi\,,\qquad\mbox{strongly in $L^p_{\text{loc}}(\R,\C^2)$},\quad \forall p\geq2\,,
\end{cases} \quad n\to\infty\,.
  \]
Observe that by \eqref{eq:PS} and \eqref{eq:levelbound} we get
\[
\lim_{n\to\infty} \frac{1}{4}\int_\R \vert\Phi^1_n\vert^4+\vert\Phi^2_n\vert^4\,dx =c(\omega)>0
\]
 and then
\be\label{eq:L4bound}
\liminf_{n\to\infty} \int_\R \vert\Phi_n\vert^4 \,dx =\alpha>0
\ee
for some $\alpha>0$. In order to obtain strong $X$-convergence for $(\Phi_n)$ we use a Concentration Compactness argument \cite{CCcompactI,CCcompactII}. Suppose, first the sequence is \emph{vanishing}, that is, for all $R>0$
\[
\limsup_{n\to\infty}\sup_{y}\int_{B_R(y)}\vert\Phi_n\vert^4\,dx=0\,.
\]
  This implies, in particular, that $\Phi_n\to0$ in $L^4$ as $n\to\infty$. But this contradicts \eqref{eq:L4bound}, so vanishing alternative is ruled out. 
  
  We are thus in the \emph{dichotomy case} (we consider \emph{concentration} as a particular case of \emph{dichotomy} with only one non trivial profile). Then there exist a sequence of points $y_n\in\R$ and radii $R^1_n,R^2_n\to\infty$ with $\lim_{n\to\infty}\frac{R^2_n}{R^1_n}=\infty$, such that
  \[
  \lim_{n\to\infty}\int_{B_{R^1_n}(y_n)}\vert\Phi_n\vert^4\,dx = \tilde{\alpha}\in(0,\alpha)\,,\quad \lim_{n\to\infty}\int_{B_{R^2_n}(y_n)\setminus B_{R^1_n}(y_n)}\vert\Phi_n\vert^4\,dx =0\,.
  \]
  Moreover, for any $\eps>0$ there exists $R_\eps>0$ such that
  \be\label{eq:L4conc}
  \int_{B_{R_\eps}(y_n)}\vert\Phi_n\vert^4\,dx \geq \tilde{\alpha}-\eps\,,\qquad\forall n\in\N\,,
  \ee
  that is, almost all the portion $\tilde{\alpha}$ of the $L^4$-norm is concentrated in the ball $B_{R_\eps}(y_n)$. Up to translation, wee can center the bump at the origin and take $y_n=0$ in the above formulas. 
  
  Consider a cutoff $\theta\in C^\infty_c([0\infty))$, with $\theta\equiv1$ on $[0,1]$, $\theta\equiv0$ on $[2,\infty)$ and $\vert\theta'\vert\leq 2$. Define the sequence
  \[
  \Psi_n:=\theta(\vert x\vert/R^1_n)\Phi_n\,.
  \]
  It is not hard to see that $\Psi_n$ is a Palais-Smale sequence for $\cL$ at level $\tilde{\alpha}>0$, and \eqref{eq:L4conc} ensures $\Psi_n\to\Psi$, strongly in $L^4$ as $n\to\infty$, for some $\Psi\in X$.We can now turn this into \emph{strong} $X$-convergence as follows.
  
  Observe that by \eqref{eq:PS} there holds
  \[
  (\cD+\beta\sigma_3-\omega)\Psi_n=G(\Psi_n)\Psi_n+o(1)\,,\qquad \mbox{in $H^{-1/2}(\R,\C^2)$}\,,
  \]
  and then
    \be\label{eq:invertDirac}
  \Psi_n=(\cD+\beta\sigma_3-\omega)^{-1}[G(\Psi_n)\Psi_n]+o(1)\,,\qquad \mbox{in $H^{1/2}(\R,\C^2)$}\,.
  \ee
  We claim that 
\be\label{eq:square}
\vert\Psi_n\vert^2\Psi_n\to\vert\Psi\vert^2\Psi\,,\qquad \mbox{strongly in $L^{4/3}(\R,\C^2)$}\,.
\ee 

To this aim, observe that, up to subsequences, 
\be\label{eq:L^2}
\vert\Psi_n\vert^2\to\vert\Psi\vert^2\,,\qquad  \mbox{strongly in $L^2(\R,\C^4)$.}
\ee
 Indeed, the sequence is bounded in $L^2$ as $\Vert\vert\Psi_n\vert^2\Vert_{L^2}=\Vert\Psi_n\Vert^2_{L^4}\leq C$, for all $n\in\N$. Then $\vert\Psi_n\vert^2\rightharpoonup \vert\Psi\vert^2$, weakly in $L^2$. Moreover, $\Vert\vert\Psi_n\vert^2\Vert_{L^2}=\Vert\Psi_n\Vert^2_{L^4}\to\Vert\Psi\Vert^2_{L^4}=\Vert\vert\Psi\vert^2\Vert_{L^2}$, and the $L^2$ strong convergence follows.
 
 Notice that
 \[
 \Vert \vert\Psi_n\vert^2\Psi_n-\vert\Psi\vert^2\Psi\Vert_{L^{4/3}}\leq  \Vert \vert\Psi_n\vert^2\Psi_n-\vert\Psi_n\vert^2\Psi_\infty\Vert_{L^{4/3}}+ \Vert \vert\Psi_n\vert^2\Psi-\vert\Psi\vert^2\Psi\Vert_{L^{4/3}}
 \]

The H\"older inequality and strong $L^4$-convergence give
\be\label{eq:1}
\begin{split}
 \Vert \vert\Psi_n\vert^2\Psi_n-\vert\psi_n\vert^2\Psi\Vert^{4/3}_{L^{4/3}}&=\int_{\R^2}\vert\vert\Psi_n\vert^2\Psi_n-\vert\Psi_n\vert^2\Psi_\infty \vert^{\frac{4}{3}}\,dx \\
  &\leq\left(\int_{\R^2}\vert\Psi_n\vert^4 \,dx\right)^{\frac{2}{3}}\left(\int_{\R^2}\vert\Psi_n-\Psi_\infty\vert^4 \,dx\right)^{\frac{1}{3}}\\
  &\leq C\left(\int_{\R^2}\vert\Psi_n-\Psi_\infty\vert^4 \,dx\right)^{\frac{1}{3}}=o(1)\,.
\end{split}
\ee
Similarly, by H\"older and \eqref{eq:L^2} we get
\be\label{eq:2}
\begin{split}
 \Vert \vert\Psi_n\vert^2\Psi-\vert\Psi\vert^2\Psi\Vert_{L^{4/3}}&=\int_{\R^2}\vert\vert\Psi\vert^2\Psi-\vert\Psi\vert^2\Psi \vert^{\frac{4}{3}}\,dx\\
  & \leq\left( \int_{\R^2}\vert\vert\Psi_n\vert^2-\vert\Psi\vert^2 \vert^2\,dx\right)^{\frac{2}{3}}\left( \int_{\R^2}\vert\Psi \vert^4\,dx\right)^{\frac{1}{3}}=o(1)\,.
\end{split}
\ee
The claim is proved combining \eqref{eq:1} and \eqref{eq:2}. Then, observing that $\vert G(\Psi_n)\vert\sim\vert\Psi_n\vert^2$, the desired strong $X$-convergence follows from \eqref{eq:invertDirac}.
\smallskip

We are now in a position to give the proof of the last main result of the paper.

    \begin{proof}[Proof of Theorem \ref{thm:main4}]
  The existence of a Palais-Smale sequence for $\cL$ is provided by Proposition \ref{prop:PS}, and the subsequent analysis shows how to find a strongly convergence subsequence. So $\cL$ admits a critical point, which is a weak solution to \eqref{eq:stationary}.
  \smallskip
  
We now prove the decay estimate \eqref{eq:decay2}. First, observe that $H^{1/2}(\R,\C^2)\hookrightarrow L^p(\R,\C^2)$, for all $p\in[2,\infty)$. Then
\be\label{eq:Dpsi}
\cD\Phi=(\omega-\beta\sigma_3)\Phi+G(\Phi)\Phi\in L^2(\R,\C^2)\,,
\ee
and since $\Vert \cD\Phi\Vert_{L^2}=\Vert\partial_x\Phi\Vert_{L^2}$, we conclude that $\Phi\in H^1(\R,\C^2)$. By the Sobolev embedding, $\Phi$ is thus H\"older continuous and tends to zero at infinity, that is
\be\label{eq:zeroatinfinity}
\lim_{x\to\pm\infty}\Phi(x)=0\,.
\ee
Since $\cD^2=-\partial_{xx}$, by \eqref{eq:Dpsi} one concludes that 
\[
-\partial_{xx}\Phi=(\omega-\beta\sigma_3)\cD\Phi-\imath\beta'\sigma_1\sigma_3\Phi+\cD(G(\Phi)\Phi)\in L^2(\R,\C^2)\,,
\]
as $\beta'\in L^\infty(\R,\C^2)$ and $\cD(G(\Phi)\Phi)$ is the sum of terms of the schematic form $\vert\Phi\vert^2\partial_x\Phi$. Then $\Phi\in H^2(\R,\C^2)$ and $\partial_x\Phi$ is H\"older continuous and tends to zero at infinity. 

Rewrite \eqref{eq:stationary} as
\[
(\cD+\beta(\infty)\sigma_3-\omega)\Phi=(\beta(\infty)-\beta)\sigma_3\Phi+\omega\Phi+G(\Phi)\Phi\,.
\]
Applying the operator $(\cD+\beta(\infty)\sigma_3+\omega)$ to both sides one gets
\[
(-\partial_{xx}+\beta(\infty)^2-\omega^2)\Phi=H(x)\,,
\]
where $H$ is a continuous function tending to zero at infinity. We can thus apply the Green's function to obtain
\[
\Phi(x)=\frac{1}{2\sqrt{\beta(\infty)^2-\omega^2}}\int_{\R}H(y)e^{-\vert x-y\vert \sqrt{\beta(\infty)^2-\omega^2}}\,dy\,,
\]
and the exponential decay follows. Observe, finally, that iterating the above argument one actually proves that $\Phi\in H^k(\R,\C^2)$, for all $k\in\N$ and so it is smooth.
    \end{proof}


\end{document}